\documentclass[10pt]{amsart}
\usepackage[utf8]{inputenc}
\usepackage{amscd}
\usepackage{amssymb}
\usepackage{pstricks}
\usepackage{empheq}
\usepackage{graphicx}
\usepackage{float}
\usepackage{todonotes}
\usepackage{subcaption}
\usepackage{xcolor}
\usepackage{textcomp}
\usepackage{cancel}
\usepackage[mathlines]{lineno}
\usepackage{comment}
\usepackage{ulem}
\newtheorem{theorem}{Theorem}
\newtheorem{proposition}{Proposition}
\newtheorem{lemma}{Lemma}
\newtheorem{definition}{Definition}

\newtheorem{remark}{Remark}

\numberwithin{equation}{section}
\numberwithin{lemma}{section}

\DeclareMathOperator{\diag}{Diag}

\title[]{Uniform $h$--dichotomies: noncritical uniformity and expansivity}

\author[Elorreaga]{Heli Elorreaga}
\author[Pe\~na]{Juan Francisco Pe\~na}
\author[Robledo]{Gonzalo Robledo}
\address{Departamento de Matem\'aticas, Universidad de Chile, Casilla 653, Santiago, Chile.}
\thanks{This research has been supported by Agency for Research and Development, ANID-Chile through the grants ANID/FONDECYT 1210733 (G. Robledo) the FONDECYT Postdoctorado project 3240354 (H. Elorreaga) and the doctoral
scholarship ANID/21210870 (J. Pe\~na).}

\begin{document}

\begin{abstract}
The property of exponential dichotomy can be seen as a generalization of the hyperbolicity condition
for non autonomous linear finite dimensional systems of ordinary differential equations. In 1978 W.A. Coppel proved that the exponential dichotomy 
on the half line is equivalent to the property of noncritical uniformity provided that a condition of bounded growth is verified.
In 2006 K.J. Palmer extended this result by proving that -also assuming the bounded growth property- the exponential dichotomy 
on the half line, noncritical uniformity and the exponential expansiveness are equivalent. The main contribution of this article is to 
generalize these results for the property of uniform $h$--dichotomy. This has been carried out due to a recent idea: 
under suitable conditions any $h$--dichotomy can be associated to a totally ordered topological group, which becomes the
additive group $(\mathbb{R},+)$ in case of the exponential dichotomy. The properties of this new group make possible such generalization. 
\end{abstract}

\maketitle

\section{Introduction}
This work revisits the property of uniform $h$--dichotomy
for non  autonomous linear systems of differential equations: 
\begin{equation}
\label{lin}
x'=A(t)x  \quad \textnormal{for any $t \in J:=(a_{0},+\infty)$},
\end{equation}
where $a_{0} \in \mathbb{R}$ or $J=\mathbb{R}$ and the map $t\mapsto A(t)\in M_{n}(\mathbb{R})$ is continuous on $J$. In particular, we will obtain new characterizations of this property. 

The main feature of the uniform $h$--dichotomy is that any of the nontrivial solutions of the system (\ref{lin}) can be decomposed into two components, namely a contractive and expansive one, whose respective convergence and
divergence rates are dominated by a function $h(\cdot)$ which has suitable properties.

The specific case $h(t)=e^{t}$ corresponds to the property of exponential dichotomy, which 
has was introduced in 1930 in the seminal work written by O. Perron \cite{Perron}. There exists several monographs \cite{coppel,Lin2,Massera,Mitropolski} devoted to the
exponential dichotomy and its applications. The $h$--dichotomies encompasses
the exponential dichotomy and a big amount of research has been focused on the task of extending the classical result
of the exponential dichotomy to an $h$--dichotomy framework. Although such a program has been very successful in some cases, the generalization of others has remained elusive. 

A recent result from one of the authors \cite{PV} has studied the contractive case of the
uniform $h$--dichotomies and constructed a totally ordered topological group $(J,\ast)$. We realized that the pro\-per\-ties of this group allow us to carry out a generalization to the $h$--dichotomy framework of a couple of well known results, namely,
the equivalence between the exponential dichotomy with the properties
of uniform noncriticality and exponential expansiveness.

\medskip

\noindent \textit{Notation:} From now on, a fundamental matrix of (\ref{lin}) will be denoted by $\Phi(t)$, while the corresponding transition matrix will be $\Phi(t,s)$. On the other hand, given
any norm $|\cdot|$ on $\mathbb{R}^{n}$, its induced matrix norm will be denoted by $||\cdot||$.

\subsection{The uniform $h$--dichotomy: some basic facts and main contribution}

\begin{definition}
\label{UHD}
Given an homeomorphism $h\colon (a_{0},+\infty)\to (0,+\infty)$,
where $h(\cdot)$ is strictly increasing and an interval $\mathcal{I}\subset J$, the linear system \eqref{lin} has a uniform $h$--dicohotomy on
the interval $\mathcal{I}\subset J$
if there exists a projector $P(t)$ and a couple of constants $K\geq 1$ and $\alpha>0$ such that
\begin{equation}
\label{Invariance}
P(t)\Phi(t,s)=\Phi(t,s)P(s) \quad \textnormal{for any $t,s\in (a,+\infty)$},
\end{equation}
and
\begin{equation}
\label{HD}
\left\{\begin{array}{rcl}
||\Phi(t,s)P(s)||  &  \leq & \displaystyle  K\left(\frac{h(t)}{h(s)}\right)^{-\alpha} \quad \textnormal{for any $t\geq s$ with $t,s\in \mathcal{I}$} \\\\
||\Phi(t,s)Q(s)||  &  \leq & \displaystyle K\left(\frac{h(s)}{h(t)}\right)^{-\alpha} \quad \textnormal{for any $s\geq t$ with $t,s\in \mathcal{I}$},
\end{array}\right.
\end{equation}
where $Q(\cdot)$ verifies $P(t)+Q(t)=I$ for any $t\in \mathcal{I}$.
\end{definition}

To the best of our knowledge, the property of uniform $h$--dichotomy has been introduced by M. Pinto
in \cite{Pinto92} and extended for a more general framework in \cite{NP94,NP95}.

Henceforth, any function $h\colon (a_{0},+\infty)\to (0,+\infty)$ having the properties stated in the above Definition will be called a \textit{growth rate}. In consequence, the Definition \ref{UHD} encompasses well known dichotomies associated to specific growth rates
such as the exponential dichotomy or the generalized exponential dichotomy (GED) which are respectively described by the growth rates
$h(t)=e^{t}$ and $h(t)=\exp(\int_{0}^{t}\gamma(\tau)\,d\tau)$ where $\gamma \colon \mathbb{R}\to \mathbb{R}$ is continuous, non negative and satisfies
\begin{displaymath}
\lim\limits_{t\to +\infty}\int_{0}^{t}\gamma(\tau)\,d\tau=+\infty \quad \textnormal{and} \quad 
\lim\limits_{t\to -\infty}\int_{t}^{0}\gamma(\tau)\,d\tau=+\infty,
\end{displaymath}
we refer the readers to \cite{Jiang,Martin,Muldowney,WXZ} for details. An illustrative example 
is given by the linear diagonal system:
\begin{displaymath}
\left(\begin{array}{c}
\dot{x}_{1}\\
\dot{x}_{2}
\end{array}\right)=
\left[\begin{array}{cc}
-\frac{1}{(1+t)\ln(1+t)}  &  0 \\
0 & \frac{1}{(1+t)\ln(1+t)}\end{array}\right]
\left(\begin{array}{c}
x_{1}\\
x_{2}
\end{array}\right) \quad \textnormal{for any $t>0$},
\end{displaymath}
where the growth rate $h\colon (0,+\infty)\to (0,+\infty)$ is $h(t)=\ln(1+t)$. In fact, we can see that the above system has a uniform $h$--dichotomy on $\mathcal{I}=(0,+\infty)$ with constants $K=\alpha=1$ and projector $P(s)=\diag\{1,0\}$ for any $s>0$.

A second illustrative example can be given by 
\begin{displaymath}
\left(\begin{array}{c}
\dot{x}_{1}\\
\dot{x}_{2}
\end{array}\right)=
\left[\begin{array}{cc}
-\frac{\alpha}{t}  &  0 \\
0 & \frac{\alpha}{t}\end{array}\right]
\left(\begin{array}{c}
x_{1}\\
x_{2}
\end{array}\right) \quad \textnormal{for any $t>0$},
\end{displaymath}
where $\alpha$ is a positive constant. In this case, the growth rate $h\colon (0,+\infty)\to (0,+\infty)$ is $h(t)=t$. It can be easily verified that the above system  has a uniform $h$--dichotomy on $\mathcal{I}=(0,+\infty)$ with constants $K=1$, $\alpha>0$ and projector $P(s)=\diag\{1,0\}$ for any $s>0$. 

From this point on, when considering the Definition \ref{UHD}, we will focus our interest on the intervals $\mathcal{I}=J$ and $\mathcal{I}=[h^{-1}(1),+\infty)$.

In order to gain a better understanding of the $h$--uniform dichotomy, let $t\mapsto x(t,t_{0},x_{0})$ be a solution of (\ref{lin}) passing through $x_{0}$ at $t=t_{0}$. Notice that, for any non trivial initial condition $x_{0}$, this solution can be decomposed as follows:
\begin{equation}
\label{splitting}
x(t,t_{0},x_{0})=\Phi(t,t_{0})x_{0}=\underbrace{\Phi(t,t_{0})P(t_{0})x_{0}}_{:=x^{+}(t,t_{0},x_{0})}+\underbrace{\Phi(t,t_{0})Q(t_{0})x_{0}}_{:=x^{-}(t,t_{0},x_{0})}.
\end{equation}

Moreover, it can be proved from (\ref{Invariance})--(\ref{HD}) that 
\begin{displaymath}
|x^{+}(t,t_{0},x_{0})|\leq K\left(\frac{h(t_{0})}{h(t)}\right)^{\alpha}  \quad \textnormal{and} \quad
\frac{1}{K}|Q(t_{0})x_{0}| \left( \frac{h(t)}{h(t_{0})}\right)^{\alpha}\leq |x^{-}(t,t_{0},x_{0})|,
\end{displaymath}
and, as $h(t)\to +\infty$ when $t\to +\infty$, we can see that any nontrivial solution of
(\ref{lin}) can be splitted in two solutions: 

\begin{itemize}
\item[$\bullet$] The solution $t \mapsto x^{+}(t,t_{0},x_{0}):=x(t,t_{0},P(t_{0})x_{0})$ is called an \textit{$h$--contraction} because converges to the origin since is upper bounded by a convergent map $t\mapsto h(t)^{-\alpha}$.

\item[$\bullet$] The solution $t \mapsto x^{-}(t,t_{0},x_{0}):=x(t,t_{0},Q(t_{0})x_{0})$ is a called an \textit{$h$--expansion} because is lowerly bounded by a divergent map $t\mapsto h(t)^{\alpha}$. 
\end{itemize}

Notice that the above decomposition (\ref{splitting}) shows that any solution of (\ref{lin}) can be splitted in two ones having a \textit{dichotomic} behavior: either $h$--contractions or $h$--expansions. This fact is behind the name $h$--dichotomy. Why are also we using the adjective \textit{uniform}? the explanation is more subtle and will require additional results that will described later.

\subsection{Main results}
The exponential dichotomy on $[0,+\infty)$ has been wi\-de\-ly studied in the literature, there exists
results devoted to admissibility, roughness, topological equivalence, characterization of its associated spectrum, etc. In particular, if the linear system (\ref{lin}) has a bounded growth on $[0,+\infty)$, namely, there exists $K_{0}>0$ and $\beta\geq 0$ such that
\begin{equation}
\label{croissance}
||\Phi(t,s)||\leq K_{0}e^{\beta(t-s)} \quad \textnormal{for any $t\geq s\geq 0$},    
\end{equation}
then it has been proved that the following properties are equivalent to the uniform exponential dichotomy on $[0,+\infty)$:

\noindent $\bullet$ The system \eqref{lin} is \textit{uniformly noncritical} on $[0,+\infty)$ if there exists $T>0$ and $\theta \in (0,1)$ such that any solution $t\mapsto x(t)$ of \eqref{lin} satisfies:
    \begin{equation}
    \label{NCP-DE}
        |x(t)|\leq \theta\sup\limits_{|u-t|\leq T}|x(u)| \quad \textnormal{for} \quad t\geq T.
    \end{equation}

\noindent $\bullet$  The system \eqref{lin} is \textit{exponentially expansive} on $[0,+\infty)$ if there exists  positive constants $L$ and $\mu$ such that if $t\mapsto x(t)$ is any solution of \eqref{lin} and $[a,b]\subset [0,+\infty)$, then for $a\leq t\leq b$
    \begin{equation}\label{e-exp}
        |x(t)|\leq L\left\{e^{-\mu(t-a)}|x(a)|+ e^{-\mu(b-t)}|x(b)|\right\}.
    \end{equation}

The above definition of noncritical uniformity has been used without being explicitly mentioned  
by W.A. Coppel in \cite[p.14]{coppel} while K.J. Palmer proposes it explicitly in \cite[Def.4]{Palmer}.
Furthermore, Palmer points out that the strong relation with a property, also called noncritical uniformity, which was previously introduced by N.N. Krasovskii in \cite[pp.57--58]{Krasovski} 
where for any $\theta \in (0,1)$ there exists $T>0$ such that (\ref{NCP-DE}) is verified.

The noncritical uniformity --in the Krasovskii's sense-- is described as a distinguished property of the solutions of nonlinear and nonautonomous systems having an equilibrium at the origin. 
In fact, it is proved \cite[Th.1]{Krasovski} that is a necessary and sufficient condition ensuring the existence of a Lyapunov function 
in a neighborhood of the origin. This property has been revisited and tailored for the linear case by J.L. Massera and J.J. Sch\"affer in \cite[p.558]{Massera} who prove its equivalence with the exponential dichotomy in an abstract context whereas the proof in a finite dimensional framework has been made by W.A. Coppel in \cite{coppel} by considering the Palmer's definition, this result has been recently extended for generalized ordinary differential equations in \cite[Th.3.8]{Bonotto}.   

As stated in \cite[p.S173]{Palmer}, the property of exponential expansiveness has its origin in the theory
of discrete dynamical systems $x_{k+1}=f(x_{k})$ where $f\colon \mathbb{R}^{n}\to \mathbb{R}^{n}$ is a diffeomorphism
and the compact set $S$ is invariant under $f$. If $S$ is also hyperbolic for $f$, then there exists positive constants $d$,
$L$ and $\lambda<1$ such that if $\{x_{k}\}_{k=a}^{b}$ and $\{y_{k}\}_{k=a}^{b}$ are orbits of $f$ with $x_{k}\in S$ such that
$|x_{k}-y_{k}|\leq d$ for $a\leq k\leq b$, then
\begin{displaymath}
|x_{k}-y_{k}|\leq L\lambda^{k-a}|x_{a}-y_{a}|+L\lambda^{b-k}|x_{b}-y_{b}|   \quad \textnormal{for $a\leq k\leq b$}.
\end{displaymath}

On the other hand, the equivalence between uniform exponential dichotomy and exponential expansiveness has been proved by K.J. Palmer
in \cite{Palmer}.

The problem of generalize the above equivalences have not been addressed for dichotomies beyond the exponential one. To the best of our knowledge, the unique effort has been made in \cite[Lemma 3.1]{WXZ} in a GED framework, where the authors introduce new definitions for bounded growth and uniform noncriticality tailored for a GED context. Moreover, they proved that these properties imply an splitting of solution as (\ref{splitting}) where the contractions and expansions are dominated by the 
GED.

As we stated above, the main contribution of the present paper is to achieve a description of the uniform $h$--dichotomy in terms of properties emulating 
(\ref{NCP-DE}) and (\ref{e-exp}). This generalization is neither straightforward nor trivial. In fact, it has been carried out by using a noticeable idea stated in \cite{PV}: the growth rate $h(\cdot)$ allows to construct a topological group $(J,\ast)$ such that $h(t\ast s)=h(t)h(s)$, which mimics a property of the exponential function as well as to construct another order in $J$. This make possible to see the right part of (\ref{HD}) as a distance between $t$ and $s$, to generalize the notion of bounded growth beyond
the additive framework and to explore its relation with the $h$--dichotomy.

\subsection{Outline} The text is organized into five sections. In section 2 we describe the group $(J,\ast)$ arising from the growth rate $h(\cdot)$, and state
some of its properties allowing to see (\ref{HD}) from another perspectives. In section 3 we define the properties of bounded growth tailored for the growth rate $h$ and study its relation with the uniform $h$--dichotomy. Section 4 provide two alternative characterizations for the uniform $h$--dichotomy, namely, a definition with constant projectors which is also equivalent to the
existence of bounded projectors in (\ref{HD}). Finally, in section 5 we introduce the notions of uniform $h$--noncriticality and
$h$--expansiveness, which encompasses (\ref{NCP-DE}) and (\ref{e-exp}) respectively. We also prove that theses properties 
are equivalent with the uniform $h$--dichotomy on $[h^{-1}(1),+\infty)$ provided that (\ref{lin}) has the bounded growth property
stated in section 3.

\section{Uniform $h$--dichotomies from a group theory perspective}

Given any growth rate $h(\cdot)$, we can define the following operation:
\begin{equation}
\label{LCI}
\begin{array}{rcl}
J \times J & \to &  J  \\
(t,s) &\mapsto &  t\ast s:=h^{-1}\left(h(t)h(s)\right).
\end{array} 
\end{equation}

The set $J$ provided with the operation $\ast$ defined above has a group
structure as stated by the following result: 
\begin{proposition}[\cite{PV}]
The pair $(J,\ast)$ is an abelian topologic group, where 
the unit element is
\begin{equation}
\label{unit}
e_{\ast}:=h^{-1}(1),
\end{equation}
and, for any $t \in J$, its inverse is defined by:
\begin{equation}
\label{inverso}
t^{\ast -1}:=h^{-1}\left(\frac{1}{h(t)}\right).
\end{equation}
\end{proposition}

We know that it is always difficult to adapt us to a new notational context.  However, considering future results, it will be very useful to check carefully the following identities, which are stated for any couple $t,s\in J$ and are a direct consequence of (\ref{LCI})--(\ref{inverso}):
\begin{subequations}
  \begin{empheq}[left=\empheqlbrace]{align}
&
h(t\ast s)=h(t)h(s) \quad \textnormal{and} \quad h(s^{\ast -1})=\frac{1}{h(s)}, \label{group0} \\
&
(t\ast s)^{\ast -1}=h^{-1}\left(\frac{1}{h(t\ast s)} \right)=h^{-1}\left(\frac{1}{h(t)h(s)} \right), \label{group1} \\
& t\ast s^{\ast -1}=t\ast h^{-1}\left(\frac{1}{h(s)}\right)= h^{-1}\left(\frac{h(t)}{h(s)}\right), 
\label{group2} \\
&
(t\ast s^{\ast -1})^{\ast -1}=h^{-1}\left(\frac{1}{h(t)h(s^{\ast -1})}\right)=h^{-1}\left(\frac{h(s)}{h(t)}\right)=s\ast t^{\ast-1}. \label{group3}
 \end{empheq}
 \end{subequations}

Notice that if $J=\mathbb{R}$ and $h(t)=e^{t}$, we recover the classical additive group 
$(\mathbb{R},+)$ and (\ref{group0}) becomes the identities $e^{t+s}=e^{t}e^{s}$
and $e^{-t}=1/e^{t}$.

The property (\ref{group2}) allow us to see the uniform $h$--dichotomy from the following perspective:
\begin{equation}
\label{HD2}
\left\{\begin{array}{rcl}
||\Phi(t,s)P(s)||  &  \leq & \displaystyle  Kh(t\ast s^{\ast -1})^{-\alpha} \quad \textnormal{for any $t\geq s$  with $t,s\in \mathcal{I}$}, \\\\
||\Phi(t,s)Q(s)||  &  \leq & \displaystyle K h(s\ast t^{\ast -1})^{-\alpha} \quad \textnormal{for any $s\geq t$ with $t,s\in \mathcal{I}$}.
\end{array}\right.
\end{equation}

In addition, given $s,t\in J$, by using (\ref{LCI}) we can verify that
\begin{subequations}
  \begin{empheq}[left=\empheqlbrace]{align}
  &
t \leq s \quad \textnormal{if and only if} \quad u\ast t \leq u \ast s \quad \textnormal{for any $u\in J$} \label{group4}, \\
&
t \leq s \quad \textnormal{if and only if} \quad t\ast u \leq s \ast u \quad \textnormal{for any $u\in J$} \label{group4b}, \\
&
t \leq s \quad \textnormal{if and only if} \quad  s^{\ast-1}\leq t^{\ast-1} \label{group4c}. 
\end{empheq}
\end{subequations}

The above properties allow us to introduce the following partial order $\leq_{\ast}$ on $J$:
\begin{equation}
\label{order}
s\leq_{\ast} t  \quad \textnormal{if and only if} \quad   e_{\ast} \leq t\ast s^{\ast -1}.    
\end{equation}

The abelian group $(J,\ast)$ together with the above relation have the property:
\begin{lemma}
The group $(J,\ast, \leq_{\ast})$ is totally ordered,
\end{lemma}

\begin{proof}
By using the fact that $h(\cdot)$ and $h^{-1}(\cdot)$ are increasing functions combined with (\ref{group0})
we have that $s\leq t$ if and only if $s\leq_{\ast} t$, which implies that $\leq_{\ast}$ is a total order.

By using (\ref{group4}) combined with the above statement we have that
\begin{displaymath}
t \leq_{\ast} s \quad \textnormal{implies that} \quad u\ast t \leq u \ast s \quad \textnormal{for any $u\in J$},
\end{displaymath}
and it follows that the group is left--ordered. Finally, as the group es abelian, it is also right--ordered and the Lemma follows.
\end{proof}

As the group is ordered, we can define an absolute value $|\cdot|_{\ast}\colon J\to [e_{\ast},+\infty)$:
\begin{equation}
\label{abs}
|t|_{\ast}=\left\{\begin{array}{ccr}
t &\textnormal{if}& e_{*}\leq t, \\
t^{\ast -1} &\textnormal{if}& \,\, t<e_{\ast},
\end{array}\right.
\end{equation}
and the identity $|t|_{\ast}=|t^{\ast-1}|_{\ast}$ is straightforward. Furthermore, as $(J,\ast)$ is abelian, the triangle inequality is satisfied:
\begin{equation}
|t\ast s|_{\ast} \leq |t|_{\ast} \ast |s|_{\ast},
\end{equation}
and we refer the reader to \cite[p.2]{Chiswell}. These facts allow to define a distance $d_{\ast} \colon (J,\ast) \times (J,\ast) \to   ([e_{\ast},+\infty),\ast)$
as follows:
\begin{equation}
\label{distance}
d_{\ast}(t,s):=|t\ast s^{\ast -1}|_{\ast}.
\end{equation}

It will be useful to note that, given $L>e_{\ast}$ it follows that
\begin{equation}
\label{EVA}
|t\ast s^{\ast -1}|_{\ast}\leq L  \iff L^{\ast-1}\leq  t\ast s^{\ast -1}\leq L.
\end{equation}

Finally, notice that (\ref{order}) combined (\ref{distance}) allow to see uniform $h$--dichotomy from the alternative perspective:
\begin{equation}
\label{HD3}
\left\{\begin{array}{rcl}
||\Phi(t,s)P(s)||  &  \leq & \displaystyle  Kh( d_{\ast}(t,s))^{-\alpha} \quad \textnormal{for any $t\geq s$ with $t,s\in \mathcal{I}$} \\\\
||\Phi(t,s)Q(s)||  &  \leq & \displaystyle K h(d_{\ast}(t,s))^{-\alpha} \quad \textnormal{for any $s\geq t \geq a$ with $t,s\in \mathcal{I}$}.
\end{array}\right.
\end{equation}

The above equations imply that the $h$--contractions $t\mapsto x^{+}(t,t_{0},x_{0})$
and the $h$--expansions $t\mapsto x^{-}(t,t_{0},x_{0})$ are also dependent of the distance
$d_{\ast}(t,t_{0})$, namely, the time elapsed between $t$ and $t_{0}$ but are independent 
of the initial time $t_{0}$. In this context, the adjective \textit{uniform} must be understood
with respect to the initial time.

\section{Bounded Growth}

The bounded growth properties are well known in the study of nonautonomous linear systems
and its definitions are stated in terms of the additive structure of $(\mathbb{R},+)$. Now, when considering the group $(J,\ast)$, we have:
\begin{definition}
\label{UBG}
The linear system \eqref{lin} has an  uniform bounded $h$--growth on $\mathcal{I}\subseteq J$ if for some $T>e_{\ast}$ there exists
$C_{T}\geq 1$ such that any solution $t\mapsto x(t)$ verifies 
\begin{equation}
\label{BG}    
|x(t)|\leq C_{T}|x(s)| \quad \textnormal{for any $t\in [s,s\ast T]\cap \mathcal{I}$}.
\end{equation}
\end{definition}

\begin{definition}
\label{UBD}
    The linear system \eqref{lin} has an uniform bounded $h$--decay on $\mathcal{I}\subset J$ if for some $T>e_{\ast}$ there exists $C_T\geq1$ such that any solution  $t\mapsto x(t)$ verifies
    \begin{equation}\label{BD}
        |x(t)|\leq C_T|x(s)| \quad \textnormal{for any $t\in[s\ast T^{\ast-1}, s]\cap \mathcal{I}$}.
    \end{equation}
\end{definition}

\begin{definition}
\label{UBGD}
    The linear system \eqref{lin} has an uniform bounded $h$--growth and $h$--decay on $\mathcal{I}\subset J$ if for some $T>e_{\ast}$ there exists $C_T\geq1$ such that any solution  $t\mapsto x(t)$ verifies
    \begin{equation}\label{BGD}
        |x(t)|\leq C_T|x(s)| \quad \textnormal{for any $t\in[s\ast T^{\ast-1}, s\ast T]\cap \mathcal{I}$}.
    \end{equation}
\end{definition}

An illustrative example of the property of uniform bounded $h$--growth on 
$J$ is given for the case of systems such that for some fixed $T>e_{\ast}$ it follows that:
$$
||A||_{\mathcal{S}^{1}}:=\sup\limits_{s\in J}\int_{s}^{s\ast T}||A(\tau)||\,d\tau < +\infty, 
$$
where $||\cdot||_{\mathcal{S}^{1}}$ can be seen as the Stepanov's norm with the additive structure of $(J,\ast)$. In fact, if $s\leq t \leq s\ast T$ then by Gronwall's Lemma it can be deduced that
\begin{displaymath}
\begin{array}{rcl}
|x(t)| &\leq &|x(s)|e^{\int_{s}^{t}||A(\tau)||\,d\tau} \\\\
&\leq & |x(s)|e^{\int_{s}^{s\ast T}||A(\tau)||\,d\tau}  \\\\
&\leq & e^{||A||_{\mathcal{S}^{1}}}|x(s)|,
\end{array}
\end{displaymath}
and the property is verified.

To the best of our knowledge, the Definition \ref{UBG} has been introduced by W. Coppel in \cite[pp.8--9]{coppel} 
in the classical additive case under the name of \textit{bounded growth}. 

In addition, in \cite{coppel} it has been proved that the bounded growth on $\mathcal{I}\subset J$ is equivalent to the existence of constants $K_{0}\geq 1$ and $\beta>0$ such that
\begin{displaymath}
||\Phi(t,s)||\leq K_{0}e^{\beta(t-s)} \quad \textnormal{for any $t\geq s$ with $t,s\in \mathcal{I}$}    
\end{displaymath}
and the proof of this equivalence uses a partition $\{[s+(k-1)T,s+kT)\}_{k\in \mathbb{Z}}$  of $\mathbb{R}$ by intervals of large $T$ in the classical additive sense. In order to generalize this equivalence for an arbitrary growth rate $h$ and an additive framework $(J,\ast)$ 
we will make the following partition of $\mathcal{I}$: 
$$
\mathcal{I}=\bigcup\limits_{k\in \mathbb{Z}}I_{k}  \quad \textnormal{where $I_{k}=[s\ast T^{\ast (k-1)},s\ast T^{\ast k})$},
$$
where $T^{\ast k}$ is
\begin{equation}
\label{puissance}
 T^{\ast k}=\left\{\begin{array}{rcl}
 \underbrace{T \ast \cdots \ast T}_{k-\textnormal{times}} &\textnormal{if}& k>0 \\
 e_{\ast}  &\textnormal{if}&  k=0 \\
 \underbrace{T^{\ast -1} \ast \cdots \ast T^{\ast -1}}_{k-\textnormal{times}} &\textnormal{if}& k<0,
 \end{array}\right.
\end{equation}
and, by using (\ref{abs}) combined with (\ref{distance}) and $T>e_{\ast}$, we can verify that the above partition of $J$ have intervals of large $T$:
\begin{displaymath}
d_{\ast}(s\ast T^{\ast k},s\ast T^{\ast (k-1)})=d_{\ast}(T\ast s\ast T^{\ast (k-1)},s\ast T^{\ast (k-1)})=T.  
\end{displaymath}

Moreover, by using (\ref{group0}), we can inductively deduce that: 
\begin{equation}
\label{potencia}
h(T^{\ast k})=h(T)^{k}.
\end{equation}

Now, the following results describe useful equivalences between the definitions \ref{UBG}, \ref{UBD} and \ref{UBGD} 
with characterizations of these properties in terms of the transition matrix and the growth rate $h(\cdot)$:


\begin{lemma}\label{Lem-Eq-BG}
The system \eqref{lin} has a uniform bounded $h$--growth on $\mathcal{I}$ if and only if there exist 
$K_{0}\geq 1$ and $\beta \geq 0$ such that
\begin{equation}
\label{equivalencia}
||\Phi(t,s)|| \leq K_{0}\left(\frac{h(t)}{h(s)}\right)^{\beta}=K_{0}h(t\ast s^{\ast -1})^{\beta} \quad \textnormal{for any $t\geq s$ with $t,s\in \mathcal{I}$.} 
\end{equation}
\end{lemma}

\begin{proof}
    First we assume that \eqref{equivalencia} holds, with constants $K_0\geq1$ and $\beta\geq 0$. Then, if $t\mapsto x(t)$ is a nontrivial solution of \eqref{lin} we consider a fixed $T>e_{\ast}$ and a couple $t,s\in \mathcal{I}$ such that $s\leq t\leq s*T$. By using (\ref{group4}) we have that
    the previous inequality implies  $t\ast s^{\ast -1}\leq T$, which combined with (\ref{group2}) and the fact that $h(\cdot)$ is strictly increasing leads to:
    \begin{align*}
        |x(t)| &\leq ||\Phi(t,s)||\cdot|x(s)|\\
        &\leq K_0[h(t*s^{*-1})]^{\beta}|x(s)|\\
        &\leq K_0[h(T)]^{\beta}|x(s)|,
    \end{align*}
and  \eqref{lin} has a uniform bounded $h$--growth on $\mathcal{I}$ with $C_T=K_0[h(T)]^{\beta}\geq1$.

    Conversely, let us suppose that the system \eqref{lin} has a uniform bounded $h$--growth on $\mathcal{I}$ with constants $T>e_{\ast}$ 
    and $C_{T}\geq 1$. Now, let us recall the partition $\{I_{n}\}_{n\in \mathbb{Z}}$ of $\mathcal{I}$ above defined and consider a solution 
 $t\mapsto x(t)$ of \eqref{lin}. Notice that for any $t\in \mathcal{I}$ there exist a unique $n\in \mathbb{Z}$ such that
 $t\in I_{n}$ or equivalently
 $s*T^{*(n-1)}\leq t \leq s*T^{*n}$. By using recursively the property (\ref{BG}) we can deduce:
    \begin{equation}
        |x(t)|\leq C_T^n|x(s)|.
    \end{equation}
    On the other hand, since $T^{*(n-1)}\leq t*s^{*-1}\leq T^{*n}$, we use the identities (\ref{group2}) and (\ref{potencia})
    together with the fact that $h(\cdot)$ is strictly increasing to deduce that:
    \begin{equation}
        h(T)^{n-1}\leq \dfrac{h(t)}{h(s)}\leq h(T)^n, 
    \end{equation}
    and thus,
    \begin{equation*}
        n-1\leq \dfrac{1}{\ln(h(T))}\ln\left(\dfrac{h(t)}{h(s)}\right)\leq n.
    \end{equation*}
    Finally, we get
    \begin{align*}
        |\Phi(t,s)x(s)| = |x(t)| &\leq C_T C^{n-1}_T|x(s)|\\
        & \displaystyle \leq C_T e^{[\ln(\frac{h(t)}{h(s)})\frac{1}{\ln(h(T))}]\ln(C_T)}|x(s)|\\
        &= C_T\left(\dfrac{h(t)}{h(s)}\right)^{\dfrac{\ln(C_T)}{\ln(h(T))}}|x(s)|.
    \end{align*}
    Hence, the property \eqref{equivalencia} holds with $K_0=C_T\geq1$ and $\beta=\dfrac{\ln(C_T)}{\ln(h(T))}\geq 0$. 
\end{proof}

\begin{lemma}\label{Lem-Eq-BD}
The system \eqref{lin} has a uniform bounded $h$--decay on $\mathcal{I}$ if and only if there exist 
$K_{0}\geq 1$ and $\beta\geq0$ such that
\begin{equation}
\label{equivalenciaBD}
||\Phi(t,s)|| \leq K_{0}\left(\frac{h(s)}{h(t)}\right)^{\beta} \quad \textnormal{for any $s\geq t$ with $t,s\in \mathcal{I}$.} 
\end{equation}
\end{lemma}

\begin{proof}
    Suppose first that \eqref{equivalenciaBD} holds with constants $K_0\geq1$ and $\beta\geq 0$.
     Then, if $t\mapsto x(t)$ is a solution of \eqref{lin} we consider a fixed $T>e_{\ast}$ and a couple $t,s\in \mathcal{I}$ such that $s\ast T^{\ast -1}\leq t\leq s$. By (\ref{group4}) we can see that this implies  $s\ast t^{\ast -1}\leq T$, which combined with (\ref{group2}) and the fact that $h(\cdot)$ is strictly increasing leads to:
    \begin{align*}
        |x(t)| &\leq ||\Phi(t,s)||\cdot|x(s)|\\
                &\leq K_0[h(s\ast t^{\ast-1})]^{\beta}|x(s)|\\
                &\leq K_0[h(T)]^{\beta}|x(s)|,
    \end{align*}
and  this shows that \eqref{lin} has a uniform bounded $h$--decay with $C_T=K_0[h(T)]^{\beta}\geq1$.
    
    Now, let us suppose that \eqref{lin} has a uniform bounded $h$--decay with constants $T>e_{\ast}$ and $C_{T}\geq 1$. Then, for any couple $t,s\in \mathcal{I}$ with $t<s$, there exists a unique $n\in \mathbb{Z}$ such that $s\ast T^{\ast-n}\leq t\leq s\ast T^{\ast-(n-1)}$, or equivalently: 
    \begin{equation}
    \label{inegalite}
    (s\ast T^{\ast-(n-1)})\ast T^{\ast-1}\leq t \leq s\ast T^{\ast-(n-1)}.
    \end{equation}

    By applying recursively the property of uniform bounded $h$--decay on the intervals $I_{n}$, we can deduce that any nontrivial solution
    of (\ref{lin}) verifies:
    \begin{equation*}
        |x(t)|\leq C_T|x(s\ast T^{\ast-(n-1)})|\leq C^2_T|x(s\ast T^{\ast-(n-2)})|\leq\cdots \leq C^n_T|x(s)|.
    \end{equation*}
    
    Now, by using (\ref{group4}) and (\ref{inegalite}), we have that
$t\ast s^{\ast -1}\in [T^{\ast-n},T^{\ast-(n-1)}]$ which is equivalent to $s\ast t^{\ast -1}\in [T^{\ast (n-1)}, T^{\ast n}]$ since (\ref{group3}), (\ref{group4b}) and (\ref{group4c}). 
This last property combined with (\ref{potencia}) and the fact that $h(\cdot)$ is strictly increasing leads to:
    \begin{equation*}
        h(T)^{n-1}\leq h(s\ast t^{\ast-1})\leq h(T)^n,
    \end{equation*}
    which implies that
    \begin{equation*}
        n-1 \leq \dfrac{1}{\ln(h(T))}\ln\left(\dfrac{h(s)}{h(t)}\right)\leq n.
    \end{equation*}
    Hence, we obtain 
    \begin{align*}
        |\Phi(t,s)x(s)|=|x(t)| &\leq C_T C^{n-1}_T|x(s)|\\
        &\leq C_T e^{\ln\left(\frac{h(s)}{h(t)}\right)^{\frac{\ln(C_T)}{\ln(h(T))}}}|x(s)|\\
        &=C_T\left(\dfrac{h(s)}{h(t)}\right)^{\frac{\ln(C_T)}{\ln(h(T))}}|x(s)|.
    \end{align*}
 Finally, \eqref{equivalenciaBD} is verified  with $K_0=C_T\geq1$ and $\beta=\dfrac{\ln(C_T)}{\ln(h(T))}\geq 0$.  
\end{proof}

\begin{lemma}
The system \eqref{lin} has a uniform bounded $h$--growth and $h$--decay on $\mathcal{I}$ if and only if there exist 
$K_{0}\geq 1$ and $\beta\geq0$ such that
\begin{equation}
\label{equivalenciaBGD}
||\Phi(t,s)|| \leq K_{0}h(|t\ast s^{\ast-1}|_{\ast})^{\beta} \quad \textnormal{for any $t, s\in \mathcal{I}$}. 
\end{equation}
\end{lemma}

\begin{proof}
    The proof follows immediately from Lemmas \ref{Lem-Eq-BG} and \ref{Lem-Eq-BD}.
\end{proof}

The above results allowed to construct alternative definitions for the bounded $h$--growth properties, which are consistent
with those established in the classical additive framework. In fact, the property (\ref{equivalenciaBD}) restricted to the group $(\mathbb{R},+)$
has been introduced by K.J. Palmer in \cite{Palmer} under the name of \textit{bounded decay} while
the property (\ref{equivalenciaBGD}) has been introduced by S. Siegmund in \cite[p.253]{Siegmund-2002} also restricted to the group $(\mathbb{R},+)$ under the name \textit{bounded growth}.

\begin{remark}
\label{DCRCI}
By following the ideas of Siegmund in \cite[p.253]{Siegmund-2002} the property of uniform bounded $h$--growth and $h$--decay
is equivalent to the uniform continuity of the solutions of \eqref{lin} with respect to initial conditions, that is, 
for each $L > e_{\ast}$ and $\varepsilon>0$ there is a corresponding $\delta:=\delta(L,\varepsilon)>0$ such that for 
$\xi_{1},\xi_{2}\in \mathbb{R}^{n}$ it follows that
\begin{displaymath}
|\xi_{1}-\xi_{2}|<\delta    \Rightarrow |\Phi(t,\tau)\xi_{1}-\Phi(t,\tau)\xi_{2}|<\varepsilon \,\, \textnormal{for any $t,\tau \in \mathcal{I}$ with $L^{\ast -1}\leq t\ast \tau^{\ast-1}\leq L$}.
\end{displaymath}
\end{remark}

Last but not least, in case that a linear system (\ref{lin}) has the properties of uniform $h$--dichotomy and uniform bounded $h$--growth and $h$--decay on $\mathcal{I}$, the following result describes the relation between its corresponding constants:
\begin{lemma}
    Assume that the system \eqref{lin} verifies  \eqref{equivalenciaBGD} with constants $K_0$, $\beta$ and has a uniform $h$--dichotomy on the interval $\mathcal{I}$ with constants $K$, $\alpha$. Then it follows that $\alpha\leq \beta$.
\end{lemma}

\begin{proof}
    Without loss of generality, we will assume that $t\geq s$. Then, using properties of matrix norms verifying $||I||=1$ and the property of $h$--dichotomy combined with \eqref{equivalenciaBGD}, we have
    \begin{align*}
        1 &= ||\Phi(t,s)\Phi(s,t)||\\
          &\leq ||\Phi(t,s)P(s)\Phi(s,t)|| + ||\Phi(t,s)(I-P(s))\Phi(s,t)||\\
          &\leq K h(t\ast s^{\ast-1})^{-\alpha}K_0h(t\ast s^{\ast-1})^{\beta} + K_0h(t\ast s^{\ast-1})^{\beta}Kh(t\ast s^{\ast-1})^{-\alpha}\\
          &= 2K_0K h(t\ast s^{\ast-1})^{\beta-\alpha} \quad \mbox{ for any } \quad t\geq s.
    \end{align*}
    Observe that if $\alpha>\beta$, by letting $t\to +\infty$ we will have  $1\leq 0$, which is a contradiction. This implies that $\alpha\leq\beta$ and hence the proof is completed. 
\end{proof}

We point out that this proof follows the lines of \cite[Lemma 1]{Castaneda} which proves a related result in a nonuniform exponential dichotomy context, we also refer the reader to \cite[Lemma 2.10]{Gallegos}.

\section{Alternative characterizations for the uniform $h$--dichotomy on $\mathcal{I}$}

The following result provides an alternative definition for a uniform $h$--dichotomy
on $\mathcal{I}$ in terms of constant projectors. In addition, this equivalence is well known for the exponential
case $h(t)=e^{t}$ and we refer the reader to \cite[Prop.7.14]{CLR} for details.

\begin{proposition}
\label{ProyCte}
The system \eqref{lin} has a uniform $h$--dichotomy on $\mathcal{I}$ if and only if, given a fundamental matrix $\Phi(t)$, there exists a constant projector $P$ and constants $K\geq 1$ and $\alpha>0$
such that:
\begin{equation}
\label{eq:2.2}
\left\{\begin{array}{rccr}
||\Phi(t)P\Phi^{-1}(s)||&\leq K \left( \frac{h(t)}{h(s)}\right)^{-\alpha} & \textnormal{if} & t\geq s,\quad t,s\in \mathcal{I}\\
||\Phi(t)(I-P)\Phi^{-1}(s)|| &\leq K \left( \frac{h(s)}{h(t)}\right)^{-\alpha} & \textnormal{if} & t \leq s, \quad t,s \in \mathcal{I}.
\end{array}\right.
\end{equation}
\end{proposition}

\begin{proof}
$(\Rightarrow$: Let $\Phi(t)$
be a fundamental matrix of (\ref{lin}), as 
(\ref{lin}) has a uniform $h$--dichotomy on $\mathcal{I}$, we have the existence of a
variable projector $P(t)$ and positive constants $K,\alpha$ such that the identity (\ref{Invariance}) and the inequalities (\ref{HD}) are verified.

Let us consider a fixed $t_{0}\in \mathcal{I}$ and notice that
\begin{displaymath}
\begin{array}{rcl}
\Phi(t,s)P(s)&=&\Phi(t,t_{0})\Phi(t_{0},s)P(s)\\\\
&=&\Phi(t)\Phi^{-1}(t_{0})P(t_{0})\Phi(t_{0})\Phi^{-1}(s)\\\\
&=&\Phi(t)P\Phi^{-1}(s) \quad \textnormal{with $P=\Phi^{-1}(t_{0})P(t_{0})\Phi(t_{0})$},
\end{array}
\end{displaymath}
while the identity $\Phi(t,s)[I-P(s)]=\Phi(t)[I-P]\Phi^{-1}(s)$ can be deduced in a similar
way and the inequalities (\ref{HD}) imply (\ref{eq:2.2}) with projection $P=\Phi^{-1}(t_{0})P(t_{0})\Phi(t_{0})$.

$\Leftarrow)$ Let us assume that for any fundamental matrix $\Phi(t)$ there exists a constant projector $P$ and positive constants $K,\alpha$ such that (\ref{eq:2.2}) is verified. By defining
$$
P(t)=\Phi(t)P\Phi^{-1}(t) \quad \textnormal{for any $t\in \mathcal{I}$},
$$
we easily verify that $P(t)$ is a projection for any $t\in \mathcal{I}$, moreover
\begin{displaymath}
\begin{array}{rcl}
P(t)\Phi(t,s)&=& \Phi(t)P\Phi^{-1}(s)\\\\
&=& \Phi(t,s)\Phi(s)P\Phi^{-1}(s)\\\\
&=& \Phi(t,s)P(s),
\end{array}
\end{displaymath}
and (\ref{Invariance}) is satisfied. Finally, a direct byproduct of the above identities are
$$
\Phi(t)P\Phi^{-1}(s)=\Phi(t,s)P(s)
\quad \textnormal{and} \quad 
 \Phi(t)[I-P]\Phi^{-1}(s)=\Phi(t,s)[I-P(s)],
$$ 
then, the inequalities (\ref{eq:2.2}) imply (\ref{HD}).
\end{proof}

The next result generalizes an interesting equivalence stated without proof by Coppel in \cite[p.11]{coppel} for the exponential
dichotomy. To the best of our knowledge, the first detailed proof has been carried out by
Bonotto, Federson and Santos in \cite[Prop.3.14]{Bonotto}. We provide a straightforward generalization
to make this article self contained:

\begin{proposition}
\label{ELC} 
The linear system \eqref{lin} has a uniform $h$--dichotomy on 
$\mathcal{I}$ if and only if given a fundamental matrix $\Phi(t)$ of \eqref{lin}
there exists a projector $P$ and positive constants $K_{1},K_{2},M$ and $\alpha$ such that: 
\begin{subequations}
  \begin{empheq}{align}
  & |\Phi(t)P\xi|\leq K_{1}\left(\frac{h(t)}{h(s)}\right)^{-\alpha}|\Phi(s)P\xi| \quad \textnormal{for $t\geq s$ and $t,s\in \mathcal{I}$}  \label{ELC1}, \\
&    |\Phi(t)[I-P]\xi|\leq K_{2}\left(\frac{h(s)}{h(t)}\right)^{-\alpha}|\Phi(s)[I-P]\xi| \quad \textnormal{for $s\geq t$ and $t,s\in \mathcal{I}$}   \label{ELC2}, \\
& ||\Phi(t)P\Phi^{-1}(t)|| \leq M\quad \textnormal{for any $t\in \mathcal{I}$},  \label{ELC3} 
\end{empheq}
\end{subequations}
where $\xi$ is an arbitrary vector.
\end{proposition}
 
\begin{proof}
Firstly, let us assume that (\ref{lin}) has a uniform $h$--dichotomy on $\mathcal{I}$. Then by 
Proposition \ref{ProyCte}, for any fundamental matrix $\Phi(t)$ there exist a projector 
$P$ and positive constants $K,\alpha$ such that (\ref{eq:2.2}) is verified. Now, by using the first inequality of (\ref{eq:2.2}) and considering any $\nu\in \mathbb{R}^{n}\setminus \{0\}$,
we have that
$$
|\Phi(t)P\Phi^{-1}(s)\nu|\leq K\left(\frac{h(t)}{h(s)}\right)^{-\alpha}|\nu|. 
$$
Then the inequality (\ref{ELC1}) follows with $K_{1}=K$ and $\nu=\Phi(s)P\xi$. The inequality (\ref{ELC2})
can be deduced in a similar way while (\ref{ELC3}) follows by considering $s=t$ in the inequality (\ref{eq:2.2}).

Secondly, let us assume that the inequalities (\ref{ELC1}),(\ref{ELC2}) and (\ref{ELC3}) are verified. Notice that
if $t\geq s$, we can consider $\xi=\Phi^{-1}(s)P\nu$ with $\nu\neq 0$ in (\ref{ELC1}) and use (\ref{ELC3}) to deduce that: 
\begin{displaymath}
\begin{array}{rcl}
|\Phi(t)P\Phi^{-1}(s)\nu| &\leq& \displaystyle K_{1}\left(\frac{h(t)}{h(s)}\right)^{-\alpha}|\Phi(s)P\Phi^{-1}(s)\nu| \\\\
&\leq & \displaystyle MK_{1}\left(\frac{h(t)}{h(s)}\right)^{-\alpha}|\nu|,
\end{array}
\end{displaymath}
which leads to $||\Phi(t)P\Phi^{-1}(s)||\leq MK_{1}\left(\frac{h(t)}{h(s)}\right)^{-\alpha}$ for any $t\geq s$ with 
$t,s \in \mathcal{I}$.

On the other hand, if $s\geq t$, we can consider $\xi=\Phi^{-1}(s)[I-P]\nu$ with $\nu\neq 0$ in (\ref{ELC2}) and use again (\ref{ELC3}) to deduce that: 
\begin{displaymath}
\begin{array}{rcl}
|\Phi(t)[I-P]\Phi^{-1}(s)\nu| &\leq& \displaystyle  K_{2}\left(\frac{h(s)}{h(t)}\right)^{-\alpha}|\Phi(s)[I-P]\Phi^{-1}(s)\nu| \\\\
&\leq & \displaystyle (||I||+M)K_{2}\left(\frac{h(s)}{h(t)}\right)^{-\alpha}|\nu|,
\end{array}
\end{displaymath}
which leads to $||\Phi(t)[I-P]\Phi^{-1}(s)||\leq (||I||+M)K_{2}\left(\frac{h(s)}{h(t)}\right)^{-\alpha}$ for any $t\leq s$ with 
$t,s \in \mathcal{I}$.

Finally, by gathering the above estimations, we obtain the inequalities (\ref{eq:2.2}) with $K=\max\{M K_{1}, (||I||+M)K_{2}\}$ and the 
uniform $h$--dichotomy follows from Proposition \ref{ELC}.

\end{proof}

We can observe that \eqref{ELC3} in Proposition \ref{ELC} combined with the proof of Proposition \ref{ProyCte}
implies that the variable projector $P(t)$ is uniformly bounded for any $t\in\mathcal{I}$.  

The next result has been stated by Coppel in \cite[pp.11--12]{coppel} in a context of exponential dichotomy.
\begin{lemma}\label{C12-C3}
If the linear system \eqref{lin} has a uniform bounded $h$--growth on $\mathcal{I}$ and there exists a non null projector
$P$ and positive constants $K_{1},K_{2}$ and $\alpha$ such that \eqref{ELC1}--\eqref{ELC2} are verified, then
there exists $M>0$ such that \eqref{ELC3} is also verified.
\end{lemma}

\begin{proof}
Firstly, let us define the positive maps:
\begin{equation}
\label{PM1}
\sigma(t)=||\Phi(t)P\Phi^{-1}(t)||  \quad \textnormal{and} \quad \rho(t)=||\Phi(t)[I-P]\Phi^{-1}(t)|| \quad \textnormal{for any $t\in \mathcal{I}$}.
\end{equation}

By using triangular inequality, it can be proved that 
\begin{displaymath}
\rho(t)-\sigma(t)\leq ||I|| \quad \textnormal{and} \quad 
\sigma(t)-\rho(t)\leq ||I||,
\end{displaymath}
which implies that
\begin{equation}
\label{PM2}
|\sigma(t)-\rho(t)|\leq ||I|| \quad \textnormal{for any $t\in \mathcal{I}$}.
\end{equation}

Now, let us consider $T>e_{\ast}$ and $\eta \neq 0$. By using (\ref{ELC1}) and (\ref{ELC2}) with $\xi=\Phi^{-1}(t)\eta$, it follows that
\begin{displaymath}
\begin{array}{rcl}
|\Phi(t\ast T)P\Phi^{-1}(t)\eta| &\leq & \displaystyle K_{1}\left(\frac{h(t\ast T)}{h(t)}\right)^{-\alpha}|\Phi(t)P\Phi^{-1}(t)\eta|\\\\
&=&
\displaystyle K_{1}h(t\ast T \ast t^{\ast -1})^{-\alpha}|\Phi(t)P\Phi^{-1}(t)\eta| \\\\
&=& 
\displaystyle K_{1}h(T)^{-\alpha}|\Phi(t)P\Phi^{-1}(t)\eta|,
\end{array}
\end{displaymath}
and
\begin{displaymath}
\begin{array}{rcl}
|\Phi(t)[I-P]\Phi^{-1}(t)\eta| &\leq & \displaystyle K_{2}\left(\frac{h(t\ast T)}{h(t)}\right)^{-\alpha}|\Phi(t\ast T)[I-P]\Phi^{-1}(t)\eta|\\\\
&=&
\displaystyle K_{2}h(t\ast T \ast t^{\ast -1})^{-\alpha}|\Phi(t\ast T)[I-P]\Phi^{-1}(t)\eta| \\\\
&=& 
\displaystyle K_{2}h(T)^{-\alpha}|\Phi(t\ast T)[I-P]\Phi^{-1}(t)\eta|,
\end{array}
\end{displaymath}
which leads to 
\begin{equation}
\label{PM3}
||\Phi(t\ast T)P\Phi^{-1}(t)||\leq K_{1}h(T)^{-\alpha}\sigma(t).
\end{equation}
and
\begin{equation}
\label{PM4}
K_{2}^{-1}h(T)^{\alpha}\rho(t) \leq ||\Phi(t\ast T)[I-P]\Phi^{-1}(t)||.
\end{equation}

Let us construct the strictly increasing map $T\mapsto\gamma(T)=K_{2}^{-1}h(T)^{\alpha}-K_{1}h(T)^{-\alpha}$
and choose $T$ big enough such that $\gamma(T)>0$. Now, notice that the above estimations imply:
\begin{displaymath}
\begin{array}{rcl}
\gamma(T) &\leq &  ||\rho^{-1}(t)\Phi(t\ast T)[I-P]\Phi^{-1}(t)||-||\sigma^{-1}(t)\Phi(t\ast T)P\Phi^{-1}(t)|| \\\\
&\leq & ||\rho^{-1}(t)\Phi(t\ast T)[I-P]\Phi^{-1}(t)+\sigma^{-1}(t)\Phi(t\ast T)P\Phi^{-1}(t)|| \\\\
&\leq & ||\rho^{-1}(t)\Phi(t\ast T,t)\Phi(t)[I-P]\Phi^{-1}(t)+\sigma^{-1}(t)\Phi(t\ast T,t)\Phi(t)P\Phi^{-1}(t)|| \\\\
&\leq & ||\Phi(t\ast T,t)||\, ||\rho^{-1}(t)\Phi(t)[I-P]\Phi^{-1}(t)+\sigma^{-1}(t)\Phi(t)P\Phi^{-1}(t)||
\end{array}
\end{displaymath}

Moreover, we know by hypothesis that (\ref{lin}) has 
a uniform bounded $h$--growth on $\mathcal{I}$. That is there exist $K_{0}\geq 1$
and $\beta\geq 0$ such that (\ref{equivalencia}) is verified and consequently
\begin{displaymath}
||\Phi(t\ast T,t)||\leq K_{0}\left(\frac{h(t\ast T)}{h(t)}\right)^{\beta}=
K_{0}h(T)^{\beta},
\end{displaymath}
since (\ref{group0}).  Now, by the above inequalities combined with (\ref{PM2}). we have that:
\begin{displaymath}
\begin{array}{rcl}
\gamma(T)K_{0}^{-1}h(T)^{-\beta} &\leq & ||\rho^{-1}(t)\Phi(t)[I-P]\Phi^{-1}(t)+\sigma^{-1}(t)\Phi(t)P\Phi^{-1}(t)|| \\\\
& = & ||\rho^{-1}(t)I+(\sigma^{-1}(t)-\rho^{-1}(t))\Phi(t)P\Phi^{-1}(t)|| \\\\
& = &  \rho^{-1}(t)||I||+|\sigma^{-1}(t)-\rho^{-1}(t)|\sigma(t) \\\\
& = &  \rho^{-1}(t)\left(||I||+|\rho(t)-\sigma(t)|\right) \\\\
&\leq & 2||I||\rho^{-1}(t),
\end{array}
\end{displaymath}
and we have that:
$$
\rho(t)\leq \frac{2||I||K_{0}h(T)^{\beta}}{\gamma(T)}.
$$
Finally by (\ref{PM2}) we have that
$$
||\Phi(t)P\Phi^{-1}(t)|| \leq ||I||\left( 1+  \frac{2K_{0}h(T)^{\beta}}{\gamma(T)}\right),
$$
and the Lemma follows since the above estimation is valid for any $t\in \mathcal{I}$.
\end{proof}

The last result provides an interesting property of the $h$--dichotomy on $\mathcal{I}=[e_{\ast},+\infty)$
and follow the lines of \cite[p.13]{coppel}:
\begin{lemma}\label{Ent-int}
    If the linear system \eqref{lin} has a uniform $h$--dichotomy on a subinterval $[T_{1}, +\infty)$ with $T_1>e_{\ast}$, then it also has a uniform $h$--dichotomy on $[e_{\ast}, +\infty)$ with the same projection $P(t)$ and the same constant $\alpha>0$. 
\end{lemma}

\begin{proof}
    Let us consider, without loss of generality, a unitary matrix. Then, given a unitary norm, the number
    \begin{equation*}
        N=e^{\int_{e_{\ast}}^{T_1}||A(u)||du}>1
    \end{equation*}
    is well defined an so
    \begin{equation}\label{NBound}
        ||\Phi(t,s)||\leq N \quad \textnormal{ for } e_{\ast}\leq s,\, t\leq T_1.
    \end{equation}
    By using \eqref{NBound} combined with the property of uniform $h$--dichotomy on $[T_1,+\infty)$ we can obtain the following estimations:
    \begin{itemize}
        \item If $e_{\ast}\leq s\leq T_1\leq t$ we have that

        \begin{align*}
            ||\Phi(t,s)P(s)|| &= ||\Phi(t,T_1)\Phi(T_1,s)P(s)||\\
                              &\leq ||\Phi(t,T_1)P(T_1)||\,||\Phi(T_1,s)||\\
                              &\leq NK\left(\frac{h(t)}{h(T_1)}\right)^{-\alpha}\\
                              &\leq NK h(T_1)^{\alpha}\left(\frac{h(t)}{h(s)}\right)^{-\alpha},
        \end{align*}
        where the last inequality follows from the fact that $h(s)^{\alpha}\geq1$.

        \item If $e_{\ast}\leq s\leq t\leq T_1$ we have that 
        \begin{align*}
            ||\Phi(t,s)P(s)|| &= ||\Phi(t,T_1)\Phi(T_1,s)P(s)||\\
                              &\leq ||\Phi(t,T_1)P(T_1)||\, ||\Phi(T_1,s)||\\
                              &\leq N ||\Phi(t,T_1)\Phi(T_1,T_1)P(T_1)||\\
                              &\leq N^2K= N^2K h(T_1)^{\alpha}h(T_1)^{-\alpha}\\
                              &\leq N^2 Kh(T_1)^{\alpha}\left(\frac{h(t)}{h(s)}\right)^{-\alpha},
        \end{align*}
        where the last inequality follows from the fact that $h(t)\leq h(T_1)$ and  $h(s)^{\alpha}\geq1$.
    \end{itemize}
    Taking $\Tilde{K}=N^2Kh(T_1)^{\alpha}$ we get that
    \begin{equation*}
        ||\Phi(t,s)P(s)||\leq \Tilde{K} \left(\frac{h(t)}{h(s)}\right)^{-\alpha} \quad \textnormal{ for any } e_{\ast}\leq s\leq t.
    \end{equation*}
    The other inequality of the uniform $h$--dichotomy can be proved similarly and the Lemma follows. 
\end{proof}

\section{Alternative characterizations for the uniform $h$--dichotomy on $[e_{\ast},+\infty)$}
This section contains the main results of this article. We firstly introduce two properties: the  \textit{uniformly $h$--noncriticality} and the \textit{$h$--expansivity},
which becomes (\ref{NCP-DE}) and (\ref{e-exp}) when $h(t)=e^{t}$. We prove that these properties are equivalent
with the uniform $h$--dichotomy on $[e_{\ast},+\infty)$ provided that the system (\ref{lin})
has a uniform bounded $h$--growth on $[e_{\ast},+\infty)$.

\begin{definition}
The system \eqref{lin} is uniformly $h$--noncritical if there exists $T>e_{\ast}$
and $\theta\in (0,1)$ such that any solution $t\mapsto x(t)$ of \eqref{lin} satisfies
    \begin{equation}
    \label{NCP}
        |x(t)|\leq \theta\sup\limits_{|u\ast t^{\ast-1}|_{\ast}\leq T}|x(u)|=  \theta\sup\{|x(u)|\colon T^{\ast-1}<u\ast t^{\ast-1}\leq T \}\,\, \textnormal{ for } t\geq T.
    \end{equation}
\end{definition}

Notice that the identity in the above Definition is a consequence of the equivalence (\ref{EVA}).
This identity will be used on several occasions later on.

\begin{definition}
    We say that the system \eqref{lin} is $h$--expansive on an interval $\mathcal{I}$ if there exists  positive constants $L$ and $\beta$ such that if $t\mapsto x(t)$ is any solution of \eqref{lin} and $[a,b]\subset \mathcal{I}$, then for $a\leq t\leq b$
    \begin{equation}\label{h-exp}
        |x(t)|\leq L[h(t\ast a^{\ast-1})^{-\beta}|x(a)| + h(b\ast t^{\ast-1})^{-\beta}|x(b)|].
    \end{equation}
\end{definition}

\subsection{Uniform $h$--dichotomy and uniform $h$--noncriticality: preliminary results }

\begin{lemma}
    If the system \eqref{lin} has a uniform $h$--dichotomy on $[e_{\ast},+\infty)$, then is uniformly $h$--noncritical on
    $[e_{\ast},+\infty)$.
\end{lemma}

\begin{proof}
    Suppose that \eqref{lin} has a uniform $h$--dichotomy on $[e_{\ast},+\infty)$ with constants $K$ and  $\alpha$. For any solution $t\mapsto x(t)$ of (\ref{lin}) let us define the auxiliary functions: 
    \begin{equation}
    \label{x1x2}
        x_1(t):= \Phi(t,s)P(s)x(s) \quad \textnormal{ and } \quad x_2(t):= \Phi(t,s)Q(s)x(s).
    \end{equation}
    Now, it can be easily verified that $x_1(t)$ and $x_2(t)$ satisfy the following properties:
\begin{subequations}
  \begin{empheq}{align}
  & x(t)=x_1(t)+x_2(t)  \label{i)}, \\
&   x_1(s)=P(s)x(s) \quad \textnormal{and} \quad x_2(s)=Q(s)x(s)  \label{ii)}, \\
& P(t)x_{1}(t)=x_{1}(t) \quad \textnormal{and} \quad Q(t)x_{2}(t)=x_{2}(t)  \label{iii)}. 
\end{empheq}
\end{subequations}
    
By using (\ref{HD2}), (\ref{x1x2}),(\ref{ii)}) and (\ref{iii)}) we can deduce the following estimations
when $t\geq s$: 
    \begin{equation*}
    \left\{\begin{array}{rcl}
       |x_1(t)| & = & |P(t)\Phi(t,s)x_1(s)|=|\Phi(t,s)P(s)x_1(s)|\leq Kh(t\ast s^{\ast -1})^{-\alpha}|x_1(s)|, 
    \\\\
        |x_2(s)| &=& |\Phi(s,t)Q(t)\Phi(t,s)Q(s)x_2(s)|\leq Kh(t\ast s^{\ast-1})^{-\alpha}|x_2(t)|.
    \end{array}\right.
    \end{equation*}

    Similarly, if $t\leq s$ we can deduce the estimations:
    \begin{equation*}
    \left\{\begin{array}{rcl}
       |x_2(t)| &=& |\Phi(t,s)Q(s)x(s)|=|\Phi(t,s)Q(s)x_2(s)|\leq Kh(s\ast t^{\ast -1})^{-\alpha}|x_2(s)|,\\\\ 
       |x_1(s)| &=& |\Phi(s,t)P(t)\Phi(t,s)P(s)x_1(s)|\leq Kh(s\ast t^{\ast-1})^{-\alpha}|x_1(t)|.
    \end{array}\right.   
    \end{equation*}
      \medskip 
    The above inequalities can be summarized as follows:
       \begin{equation*}
               |x_1(t)|\leq Kh(t\ast s^{\ast -1})^{-\alpha}|x_1(s)| \quad \textnormal{and} \quad |x_2(t)|\geq K^{-1}h(t\ast s^{\ast-1})^{\alpha}|x_2(s)| \,\,\, \textnormal{if} \,\, t\geq s
       \end{equation*}
        and 
       \begin{equation*}
           |x_2(t)|\leq Kh(s\ast t^{\ast -1})^{-\alpha}|x_2(s)|;\quad \textnormal{and}\quad |x_1(t)|\geq K^{-1}h(s\ast t^{\ast-1})^{\alpha}|x_1(s)| \,\,\, \textnormal{if}\,\, t\leq s.
       \end{equation*}
       
       Now, we will assume that $|x_2(s)|\geq |x_1(s)|$, then if $t\geq s$,  the above inequalities and (\ref{i)}) imply that

       \begin{align*}
           |x(t)| &\geq |x_2(t)| - |x_1(t)|\\
           &\geq K^{-1}h(t\ast s^{\ast-1})^{\alpha}|x_2(s)| - Kh(t\ast s^{\ast-1})^{-\alpha}|x_1(s)|\\
           &\geq \left\{K^{-1}h(t\ast s^{\ast-1})^{\alpha} - Kh(t\ast s^{\ast-1})^{-\alpha}\right\}|x_2(s)|. 
       \end{align*}

       Similarly, if $|x_2(s)|\leq |x_1(s)|$, then for $s\geq t$,
       
       \begin{align*}
        |x(t)| &\geq |x_1(t)| - |x_2(t)|\\
        &\geq K^{-1}h(s\ast t^{\ast-1})^{\alpha}|x_1(s)| - Kh(s\ast t^{\ast-1})^{-\alpha}|x_2(s)|\\
        &\geq \left\{K^{-1}h(s\ast t^{\ast-1})^{\alpha} - Kh(s\ast t^{\ast-1})^{-\alpha}\right\}|x_1(s)|.
       \end{align*}
       
       Defining the auxiliary function
       \begin{equation*}
           u\mapsto \Psi(u):= K^{-1}h(u)^{\alpha} - K h(u)^{-\alpha},
       \end{equation*}
    it follows that
    \begin{equation}\label{In-psi1}
        |x(t)|\geq \Psi(t\ast s^{\ast-1})|x_2(s)| \quad \textnormal{ if } t\geq s,
    \end{equation}
    and
    \begin{equation}\label{In-psi2}
        |x(t)|\geq \Psi(s\ast t^{\ast-1})|x_1(s)| \quad \textnormal{ if } t\leq s.
    \end{equation}
    For the function $\Psi$, we can deduce that
    \begin{itemize}
        \item $\Psi(e_{\ast}) = K^{-1}h(e_\ast)^{\alpha} - K h(e_\ast)^{-\alpha} = \dfrac{1}{K} - K \leq 0$.
        \item $\Psi$ is strictly increasing since  $h$  is positive and strictly increasing.
        \item $\Psi$ is upperly unbounded.
    \item $\Psi(t_0)=0$, where $t_0=h^{-1}\Big(e^{\frac{1}{\alpha}\ln K}\Big)>h^{-1}(1)=e_{\ast}$ and $K>1$.
    \end{itemize}
For any $\theta\in (0,1)$ we can choose $T>h^{-1}\Big(e^{\frac{1}{\alpha}\ln K}\Big)$ such that $\Psi(T)\geq\dfrac{2}{\theta}$. Then if $t\ast s^{\ast-1}\geq T>e_{\ast}$, from \eqref{In-psi1}, we have
\begin{align*}
    |x(t)| \geq \Psi(T)|x_2(s)|\geq\dfrac{2}{\theta}|x_2(s)|.
\end{align*}
Moreover, as $t\ast s^{\ast-1}\geq T>e_{\ast}$ implies that $t>s$, we will consider the case
$|x_1(s)|\leq |x_2(s)|$ and we get
\begin{equation*}
    |x(s)|\leq |x_1(s)| + |x_2(s)| \leq 2|x_2(s)|\leq \theta|x(t)|.
\end{equation*}
In a similar way, if $s\ast t^{\ast-1}\geq T>e_{\ast}$, from \eqref{In-psi2} and the inequality $|x_2(s)|\leq |x_1(s)|$, it is shown that
\begin{equation*}
    |x(t)| \geq \dfrac{2}{\theta}|x_1(s)| \geq \dfrac{1}{\theta}|x(s)|.
\end{equation*}
Thus, by gathering the above estimations, we can conclude that
\begin{equation*}
    |x(s)| \leq \theta|x(t)|, \quad \textnormal{ if } \quad |s\ast t^{\ast-1}|_{\ast}\geq T.
\end{equation*}
Next, for a fixed $s\geq T$, we choose the limit case $t=s\ast T\geq s$, then
\begin{align}
    |x(s)| &\leq \theta|x(s\ast T)|\nonumber\\
    &\leq \theta\sup\limits_{s\leq u\leq s\ast T}|x(u)|=\theta \sup\{|x(u)| : e_{\ast}\leq u\ast s^{\ast-1}\leq T\}.\label{NCmayor}
\end{align}
For a fixed $s\geq T$, we now choose $t=s\ast T^{\ast-1}<s$, then 
\begin{align}
    |x(s)|&\leq \theta|x(s\ast T^{\ast-1})|\nonumber\\
    &\leq \theta\sup\limits_{s\ast T^{\ast-1}\leq u\leq s}|x(u)| = \theta \sup\{|x(u)| : T^{\ast-1}\leq u\ast s^{\ast-1}\leq e_{\ast}\}.\label{NCmenor}
\end{align}
Finally, \eqref{NCmayor} and \eqref{NCmenor} imply that
\begin{equation*}
    |x(s)|\leq \theta\sup\limits_{|u\ast s^{\ast-1}|_{\ast}\leq T}|x(u)| \quad \textnormal{ for } s\geq T,
\end{equation*}
which concludes the proof of the Lemma.
\end{proof}

Despite that there are no converse result for the above one, we point out that when the linear system (\ref{lin}) is uniformly $h$--noncritical and also has the property of uniform bounded $h$--growth on $[e_{*},+\infty)$ with the same constant $T>e_{\ast}$, then the uniform $h$--dichotomy is verified:
\begin{theorem}\label{Lem-UNC,BG-hD}
 Assume that there exists constants $T>e_{\ast}$, $C_T>1$ and $0<\theta<1$ such that every solution $t\mapsto x(t)$ of \eqref{lin} satisfies 
\begin{subequations}
  \begin{empheq}{align}
&   |x(t)|\leq C_T|x(s)| \quad \mbox{ for } \quad e_{\ast}\leq s\leq t\leq s\ast T \label{BGT1}, \\
&   |x(t)|\leq \theta \sup\limits_{|u\ast t^{\ast-1}|_{\ast}\leq T}|x(u)| \quad \mbox{ for } t\geq T,  \label{UNCT1} 
\end{empheq}
\end{subequations}
then the system \eqref{lin} has a uniform $h$--dichotomy on $[e_{\ast},+\infty)$.
\end{theorem}

The proof will be consequece of several Lemmas. Firstly, notice that any nontrivial solution $t\mapsto x(t)$ of (\ref{lin}) is either
bounded or unbounded on $[e_{\ast},+\infty)$. The idea is to prove that bounded and unbounded solutions
has the same qualitative properties that the functions described in the splitting of solution given in (\ref{splitting}).

\begin{lemma}\label{Bound-sol}
Under the assumptions of Theorem \ref{Lem-UNC,BG-hD}, there exists $K\geq 1$ and $\alpha>0$ such that any nontrivial solution $t\mapsto x(t)$ which is bounded on $[e_{\ast},+\infty)$ verifies:
\begin{equation}
\label{contra}
        |x(t)|    \leq  \displaystyle  K\left(\dfrac{h(t)}{h(s)}\right)^{-\alpha}|x(s)| \quad \textnormal{for any $t\geq s \geq e_{\ast}$}.
    \end{equation}
\end{lemma}    

\begin{proof}

\noindent \textit{Step 1: A first estimation.}  As (\ref{BGT1}) provides information about the solution $x(\cdot)$ on $[s,s\ast T]$, we will study its properties when $t\geq s\ast T$. In order to do that let us define the map $\mu\colon [e_{\ast},+\infty)\to [0,+\infty)$ as follows: 
    \begin{equation*}
        \mu(s):=\sup\limits_{u\geq s}|x(u)|.
    \end{equation*}

Notice that $t\geq s\ast T$ is equivalent to
$t\ast T^{\ast -1}\geq s$ and also implies that $t\geq T$ since $s\geq e_{\ast}$. These facts combined with (\ref{UNCT1}) imply that:  
\begin{equation}
\label{e1}
    \begin{array}{rl}
        |x(t)| &\leq \theta\sup\limits_{|u\ast t^{\ast-1}|_{\ast}\leq T}|x(u)|\\
        &= \theta \sup\{|x(u)| : t\ast T^{\ast-1}\leq u \leq t\ast T\}\\
        &\leq  \theta \sup\limits_{s\leq u \leq t\ast T}|x(u)|\\
        &\leq \theta \sup\limits_{u\geq s}|x(u)|=\theta \mu(s).
    \end{array}
\end{equation}    

    The above estimation implies that 
    \begin{equation}
    \label{e0}
        \mu(s)=\sup\limits_{s\leq u\leq s\ast T}|x(u)|. 
    \end{equation}
    
    Indeed, by using $\sup(A\cup B)=\max\{\sup(A),\sup(B)\}$ and the estimation (\ref{e1}),
    we can deduce that
    \begin{displaymath}
        \mu(s) =\max\left\{\sup\limits_{s\leq u\leq s\ast T}|x(u)|,\sup\limits_{s\ast T \leq u}|x(u)|\right\}\leq  \max\left\{\sup\limits_{s\leq u\leq s\ast T}|x(u)|,\theta \mu(s)\right\},
    \end{displaymath}
which leads to 
$$
\mu(s)\leq \sup\limits_{s\leq u\leq s\ast T}|x(u)|,
$$
whereas the inverse inequality is trivial and (\ref{e0}) follows.

    Therefore, by using (\ref{BGT1}) and recalling that $t\geq s\ast T$, we deduce that
    \begin{align*}
        |x(t)| &\leq \theta \mu(s)\\
        &= \theta \sup\limits_{s\leq u\leq s\ast T}|x(u)|\\
        &\leq \theta C_T |x(s)|,
    \end{align*}
    which combined with (\ref{BGT1}) implies
    \begin{equation*}
        |x(t)| \leq C_T|x(s)| \quad \mbox{ for } \quad e_{\ast}\leq s\leq t. 
    \end{equation*}

\noindent \textit{Step 2: A useful inequality.} At this step, we will assume that
 $t\geq s\ast T^{\ast n}$ for some $n\in \mathbb{N}$. In addition,  by recalling that $T^{\ast  n}$ is strictly increasing, we have that $t\geq T$ and $t\ast T^{\ast -1}\geq s\ast T^{\ast (n-1)}$. These inequalities and (\ref{UNCT1}) imply: 
  \begin{equation}
  \label{e2}
    \begin{array}{rl}
        |x(t)| &\leq \theta \sup\limits_{|u\ast t^{*-1}|_{\ast}\leq T}|x(u)|\\
        &\leq \theta \sup\limits_{u\geq s\ast T^{\ast (n-1)}}|x(u)|=\theta \mu(s\ast T^{\ast(n-1)}).
    \end{array}
   \end{equation} 
    
As $t\geq s\ast T^{\ast n}$ and noticing that $n\in \mathbb{N}$ is arbitrary, the above estimation implies that 
    \begin{equation}
    \label{e3}
        \mu(s\ast T^{*n}) \leq \theta \mu(s\ast T^{\ast (n-1)}).
    \end{equation}

\noindent \textit{Step 3: $t\mapsto |x(t)|$ is an $h$--contraction.}

At this step we will assume that $t\geq s$. Notice there always exists some $n\in \mathbb{N}$ such that $s\ast T^{\ast n}\leq t\leq s\ast T^{\ast (n+1)}$, then the estimations (\ref{e2})--(\ref{e3}) together with $T^{\ast 0}=e_{\ast}$ as we defined in (\ref{puissance}) allow us to deduce that:
    \begin{displaymath}
        |x(t)|  \leq \theta \mu(s\ast T^{\ast (n-1)})\leq \theta^2 \mu(s\ast T^{\ast(n-2)})\\
        \cdots \leq \theta^n\mu(s),
    \end{displaymath}
and, by using (\ref{BGT1}) and (\ref{e0}), we have that

\begin{equation}  
\label{e5}
|x(t)|\leq  \theta^n\sup\limits_{s\leq u\leq s\ast T}|x(u)|\leq \theta^n C_T|x(s)|.
    \end{equation}
    
    On the other hand, as $T^{\ast n}\leq t\ast s^{\ast -1}\leq T^{\ast(n+1)}$, we use (\ref{LCI}) and (\ref{group0}) to see that 
    \begin{equation*}
        n\leq \dfrac{1}{\ln(h(T))}\ln\left(\dfrac{h(t)}{h(s)}\right)\leq n+1.
    \end{equation*}
    and then:
    \begin{equation*}
        \theta^{n+1}\leq \theta^{\dfrac{1}{\ln(h(T))}\ln\left(\dfrac{h(t)}{h(s)}\right)}\leq \theta^n.
    \end{equation*}

    The above estimation together with (\ref{e5}) implies:
    \begin{align*}
        |x(t)| &\leq \theta^{-1}C_T \theta^{n+1}|x(s)|\\
        &\leq \theta^{-1}C_T\left(\dfrac{h(t)}{h(s)}\right)^{\frac{\ln(\theta)}{\ln(h(T))}}|x(s)|.
    \end{align*}
    
    Taking $K=\theta^{-1}C_T\geq1$ and $\alpha=-\dfrac{\ln(\theta)}{\ln(h(T))}>0$ since $\theta \in (0,1)$ and $h(T)>1$, we obtain (\ref{contra}) and the Lemma follows.
\end{proof}

\begin{lemma}\label{Unbound-sol}
Under the assumptions of Theorem \ref{Lem-UNC,BG-hD}, there exists $K\geq 1$, $\alpha>0$ and $T_{1}>e_{\ast}$ such that any solution $t\mapsto x(t)$
unbounded on $[e_{\ast},+\infty)$ verifies:
 \begin{equation}
 \label{expa}
        |x(t)|\leq K\left(\frac{h(s)}{h(t)}\right)^{-\alpha}|x(s)| \quad \mbox{ for } \quad s\geq t\geq T_{1}> e_{\ast}.
        \end{equation}
\end{lemma}

\begin{proof}

\noindent \textit{Step 1: Preliminaries.}
    Let $t\mapsto x(t)$ be an unbounded solution with $|x(e_{\ast})|=1$. As $x(\cdot)$ is also continuous, we can take a sequence $\{t_n\}$ with $t_n>e_{\ast}$ such that  
    \begin{equation}
    \label{QPS}
        |x(t_n)|=\theta^{-n}C_T \quad \textnormal{and} \quad |x(t)|<\theta^{-n}C_T \quad \mbox{ for } \quad e_{*}\leq t< t_nm
    \end{equation}
which implies that the sequence is strictly increasing and upper unbounded. On the other hand, by the uniform bounded $h$--growth (\ref{BGT1}) with $s=e_{\ast}$ we have that:
    \begin{equation*}
        |x(t)|\leq C_T|x(e_{\ast})|=C_T<\theta^{-1}C_{T} \quad \mbox{ for } \quad e_{\ast}\leq t\leq T,
    \end{equation*}
    which allow us to see that $T<t_{1}$. 

\noindent \textit{Step 2:  $t_{n+1}\leq t_n\ast T$ for any $n\in \mathbb{N}$.}
    For if not, there exists $n_0 \in \mathbb{N}$ such that $t_{n_0+1}>t_{n_0}\ast T$, this fact combined with recalling $t_{n_0}>T$ and (\ref{UNCT1}) leads to:
    \begin{equation*}
        |x(t_{n_0})|\leq \theta\sup\limits_{|u\ast t_{n_0}^{\ast-1}|_{\ast}\leq T}|x(u)|\leq \theta\sup\limits_{e_{\ast}\leq u\leq t_{n_0}\ast T}|x(u)|\leq  \theta\sup\limits_{e_{\ast}\leq u\leq t_{n_0+1}}|x(u)|.
    \end{equation*}

    In addition, if $\sup\limits_{e_{\ast}\leq u\leq t_{n_0+1}}|x(u)|=|x(u_0)|$ where $e_{\ast}\leq u_0\leq t_{n_0+1}$, the 
    properties (\ref{QPS}) imply that
    \begin{equation*}
        |x(u_0)|<\theta^{-(n_0+1)}C_T=\theta^{-1}|x(t_{n_0})|.
    \end{equation*}

    By gathering the above estimations, we obtain that
    \begin{equation*}
        |x(t_{n_0})|\leq \theta\sup\limits_{e_{\ast}\leq u\leq t_{n_0+1}}|x(u)|<|x(t_{n_0})|,
    \end{equation*}
    which is a contradiction. 

    \medskip
    
\noindent \textit{Step 3: End of proof.}    
    Suppose that $e_{\ast}\leq t\leq s$ and $t_m\leq t<t_{m+1}$, $t_n\leq s< t_{n+1}$ with $1\leq m<n$. By (\ref{QPS})
    and the qualitative properties of $\{t_{n}\}_{n}$  we have that:
   \begin{equation}
   \label{SEF}
    \begin{array}{rl}
        |x(t)| &\leq \theta^{-m-1}C_T\\
        &=\theta^{n-m}|x(t_{n+1})|\\
        &\leq \theta^{-1}C_T \theta^{n-m+1}|x(s)|,
    \end{array}
   \end{equation} 
where the last estimation follows from $t_{n+1}\leq t_{n}\ast T$ combined with (\ref{BGT1}).

    In order to obtain a better estimation of (\ref{SEF}), we have that
    \begin{equation*}
        s\ast t^{\ast-1}\leq t_{n+1}\ast t^{\ast-1}_m\leq (t_n\ast T)\ast t^{\ast-1}_m = (t_n\ast t^{\ast-1}_m)\ast T.
    \end{equation*}
    Moreover,
    \begin{equation*}
        t_n\leq t_{n-1}\ast T\leq t_{n-2}\ast T^{*2}\leq \cdots \leq t_m\ast T^{\ast(n-m)},
    \end{equation*}
    which implies that 
    \begin{equation*}
        h(s\ast t^{*-1}) 
 \leq h(T^{\ast (n-m+1)})=h(T)^{n-m+1}
    \end{equation*}
    and so
    \begin{equation*}
        \theta^{n-m+1}\leq \theta^{\frac{\ln(h(s\ast t^{\ast-1}))}{\ln(h(T))}}.
    \end{equation*}
    Next, upon inserting the above estimation in (\ref{SEF}), we get that
    \begin{align*}
        |x(t)| &\leq \theta^{-1}C_T\theta^{\frac{\ln(h(s\ast t^{\ast-1}))}{\ln(h(T))}}|x(s)|\\
        &= \theta^{-1}C_Te^{\ln(h(s\ast t^{\ast-1}))\frac{\ln(\theta)}{\ln(h(T))}}|x(s)|\\
        &= \theta^{-1}C_T h(s\ast t^{\ast-1})^{\frac{\ln(\theta)}{\ln(h(T))}}|x(s)|,
    \end{align*}
then (\ref{expa}) is obtained with $K=\theta^{-1}C_T\geq1$, $\alpha=-\ln(\theta)/\ln(h(T))>0$
    and $T_{1}=t_{1}$.
\end{proof}

\subsection{Uniform $h$--dichotomy and uniform $h$--noncriticality: Proof of Theorem \ref{Lem-UNC,BG-hD}} Let us consider the subspace: 
$$
V_1:=\{\xi\in\mathbb{R}^n : 
 \sup\limits_{t\geq e_{\ast}} |\Phi(t,e_{\ast})\xi|<+\infty\}
 $$ 
 and let $V_2$ be a subspace of $\mathbb{R}^n$ such that $\mathbb{R}^n=V_1\oplus V_2$. For any $\xi\in V_2$, with $|\xi|=1$, let $t\mapsto x(t)=x(t,\xi)$ be the solution of \eqref{lin} with initial condition $x(e_{\ast})=\xi$. Then $t\mapsto x(t,\xi)$ is unbounded and so there exists a value $t_1:=t_1(\xi)$ such that 
 $$
|x(t,\xi)|<\theta^{-1}C_{T} \quad \textnormal{for any $e_{\ast}\leq t<t_{1}(\xi)$} \quad \textnormal{and} \quad |x(t_1,\xi)|=\theta^{-1}C_{T}.
 $$
 
 We will show that the set $B:=\{t_1(\xi): \xi\in V_2 \quad \textnormal{ and } \quad |\xi|=1\}$ is bounded. In fact, if $B$ is unbounded, there is a sequence $\{\xi_{\nu}\}_{\nu\in\mathbb{N}}\subset V_2$ with $|\xi_{\nu}|=1$ and $t_1^{(\nu)}=t_1(\xi_{\nu})\to +\infty$ as $\nu\to +\infty$. By the compactness of the unit sphere in $V_2$, we may assume that $\xi_{\nu}\to \xi_0$ as $\nu\to +\infty$, for some unit vector $\xi_0\in V_2$. Thus, from Remark \ref{DCRCI}, we have that the uniform bounded $h$--growth property implies that
 \begin{equation*}
     x(t,\xi_{\nu}) \to x(t,\xi) \quad \textnormal{ as } \quad \nu\to +\infty,
 \end{equation*}
 for all $t\geq e_{\ast}$. Since $|x(t,\xi_{\nu})|<\theta^{-1}C_{T}$ for $e_{\ast}\leq t< t_1^{(\nu)}$, it follows that
 \begin{equation*}
     |x(t,\xi_0)|\leq \theta^{-1}C_{T}, \quad \textnormal{ for } \quad e_{\ast}\leq t<+\infty,
 \end{equation*}
 which is a contradiction  with the fact that $\xi_0\in V_2$. Then there is $T_1>e_{\ast}$ such that $t_1(\xi)\leq T_1$ for all unit vector $\xi\in V_2$. Next, for all $\xi\in V_2$, with $\xi\neq 0$, Lemma \ref{Unbound-sol} implies that 
 \begin{equation*}
     |x(t,\xi)|\leq|\xi|K\left(\frac{h(s)}{h(t)}\right)^{-\alpha}|x(s,\frac{\xi}{|\xi|})|=K\left(\frac{h(s)}{h(t)}\right)^{-\alpha}|x(s,\xi)|,
 \end{equation*}
 for $s\geq t\geq T_1>e_{\ast}$.
 
 Now, let $P$ be the projection from the split $\mathbb{R}^{n}=V_1\oplus V_2$ on the subspace $V_1$. Then for each $\xi\in \mathbb{R}^n$, by Lemmas  \ref{Bound-sol} and \ref{Unbound-sol}, we have that 
 \begin{equation*}
     |\Phi(t)P\xi|\leq K\left(\frac{h(t)}{h(s)}\right)^{-\alpha}|\Phi(s)P\xi| \quad \textnormal{ for } \quad t\geq s\geq T_1 
 \end{equation*}
 and
 \begin{equation*}
     |\Phi(t)(I-P)\xi|\leq K\left(\frac{h(s)}{h(t)}\right)^{-\alpha}|\Phi(s)(I-P)\xi| \quad \textnormal{ for } \quad s\geq t\geq T_1.
 \end{equation*}
 The uniform bounded $h$--growth property together with Lemma \ref{C12-C3} imply that there exists a positive constant $M$ such that 
 \begin{equation*}
     |\Phi(t)P\Phi^{-1}(t)|\leq M \quad \textnormal{for any $t\geq T_{1}$}.
 \end{equation*}
 Hence, from Proposition \ref{ELC} we have that \eqref{lin} has a uniform $h$--dichotomy on the interval $[T_1,+\infty)$ and by Lemma \ref{Ent-int}, it has a uniform $h$--dichotomy on $[e_{\ast}, +\infty)$.

\subsection{Uniform $h$--dichotomy and $h$--expansiveness}

The following result ge\-ne\-ralizes a result of Palmer \cite[Th.1]{Palmer},
which states the equivalence between the properties of uniform $h$--dichotomy, 
uniform $h$--non criticality and $h$--expansiveness, provided that the uniform bounded $h$--growth property is verified.

\begin{theorem}
    Assume that system \eqref{lin} has a uniform bounded $h$--growth on $[e_{\ast},+\infty)$. Then the following three statements are equivalent:
    \begin{enumerate}
        \item[(i)] System \eqref{lin} has a uniform $h$--dichotomy on $[e_{\ast},+\infty)$.
        \item[(ii)] System \eqref{lin} is $h$--expansive on $[e_{\ast},+\infty)$.
        \item[(iii)] System \eqref{lin} is uniformly $h$--noncritical on $[e_{\ast},+\infty)$.
    \end{enumerate}

    Moreover, without the assumption of uniform bounded $h$--growth, it is still true that the following implications are verified: $\textnormal{(i)}\Rightarrow\textnormal{(ii)}\Rightarrow\textnormal{(iii)}$.
\end{theorem}

\begin{proof}
    We firstly assume that \eqref{lin} does not has the property of
    uniform bounded $h$--growth on $[e_{\ast},+\infty)$. Now, let us suppose that
    the system has a uniform $h$--dichotomy on $[e_{\ast},+\infty)$ with constants $K\geq1$ and $\alpha>0$. If $t\mapsto x(t)$ is a solution of \eqref{lin} and $a,b\in [e_{\ast},+\infty)$, then we have:    
    \begin{equation}\label{Th-1}
        x(t)=\Phi(t,a)P(a)x(a) + \Phi(t,a)Q(a)x(a)
    \end{equation}
    and
    \begin{equation}\label{Th-2}
        x(t)=\Phi(t,b)P(b)x(b) + \Phi(t,b)Q(b)x(b).
    \end{equation}
    
    Setting $t=a$ in \eqref{Th-2} and substituting into the second term of \eqref{Th-1} we obtain
\begin{equation*}
    x(t) = \Phi(t,a)P(a)x(a) + \Phi(t,b)Q(b)x(b).
\end{equation*}
Hence, if $e_{\ast}\leq a\leq t\leq b$, the uniform $h$--dichotomy implies that:
\begin{equation}
\label{he}
    |x(t)|\leq K\{ h(t\ast a^{\ast-1})^{-\alpha}|x(a)| + h(b\ast t^{\ast-1})^{-\alpha}|x(b)|\},
\end{equation}
which shown that (ii) holds, that is, the $h$--expansiveness is verified with
constants $K$ and $\alpha$.

Next, suppose that (ii) is satisfied as in (\ref{he}) and consider
$T$ with $t\geq T>e_{\ast}$. By using this inequality combined with the properties (\ref{group4})--(\ref{group4b}) it follows that $e_{\ast}>T^{\ast-1}$ and also $t\ast T>t>t\ast T^{\ast-1}\geq e_{\ast}$.

By considering $b=t\ast T$ and $a=t\ast T^{\ast-1}$, the estimation (\ref{he}) becomes
\begin{align*}
    |x(t)|&\leq Kh(T)^{-\alpha}\{|x(t\ast T^{\ast-1})| + |x(t\ast T)|\} \quad \textnormal{with $t>T$}\\
    &\leq 2Kh(T)^{-\alpha}\sup\{|x(u)|\colon t\ast T^{\ast-1}\leq u \leq t\ast T\}  \quad \textnormal{with $t>T$}\\
    &\leq  \theta \sup\limits_{|u\ast t^{\ast-1}|_{\ast}\leq T}|x(u)|  \quad \textnormal{with $t>T$},
\end{align*}
where $\theta=2Kh(T)^{-\alpha}$. Since $h(\cdot)$ is strictly increasing, we have that $\theta<1$ if $T$ is sufficiently large and the property (iii) follows.

Now, we assume that \eqref{lin} has a uniform bounded $h$-growth on $[e_{\ast},+\infty)$. From Theorem \ref{Lem-UNC,BG-hD} it follows that (iii)$\Rightarrow$(i) and theorem follows.
\end{proof}

\end{document}